\documentclass[12pt]{article}
\usepackage{amsmath}
\usepackage{amsfonts}
\usepackage{amssymb}
\usepackage{graphicx}
\providecommand{\U}[1]{\protect\rule{.1in}{.1in}}
\providecommand{\U}[1]{\protect\rule{.1in}{.1in}}
\newtheorem{theorem}{Theorem}

\newtheorem{corollary}[theorem]{Corollary}

\newtheorem{lemma}[theorem]{Lemma}

\newtheorem{proposition}[theorem]{Proposition}

\newenvironment{proof}[1][Proof]{\noindent\textbf{#1.} }{\ \rule{0.5em}{0.5em}}
\begin{document}

\author{G. Steinbrecher, N. Pometescu \\Department of Physics, University of Craiova, \\Str. A. I. Cuza, No.13, 200585 - Craiova, Romania}
\title{Simplified convergence proof of B\'{e}zier finite elements on D-dimensional simplex}
\date{}
\maketitle

\begin{abstract}
By using a general formalism, we expose a simplified proof of the convergence
of the B\'{e}zier polynomials attached to a continuous function defined in
arbitrary dimensional simplex. We obtain an error estimate that contains the
error in approximating by exponential functions. Our new proof is based on the
topological Stone-Weierstrass theorem.

\end{abstract}

\section{Introduction \ }

For\ the numerical simulation of the wall touching kink modes (WTKM) in the
thin wall approximation \cite{AtanasiuZaharov1} one of the problem is to
extend the existing code, designed for the smooth tokamak wall using
triangular finite elements, to the case when to the tokamak wall is attached a
limiter. In order to\ reduce the modification in the previously elaborated
programme and input data, in the generic cases we are forced to solve the
problem of nonconforming finite elements. Despite the\ construction of the
finite elements on the limiter alone does not pose complex problems, on the
contact line on the tokamak wall special problems appears due to nonconforming
position of the finite element simplexes. In this work we study the
convergence rate of B\'{e}zier simplicial finite elements,\ in order to
elaborate the mathematical foundations for the forthcoming work intended to
elaborate codes with non conforming finite elements.\ We give a new, simplest
proof of the convergence of Bernstein polynomials toward a $D$ variable
continuos function defined on\ a $D$ dimensional simplex and estimate the
speed of convergence.

\section{ \ \ The B\'{e}zier polynomials on D dimensional simplex}

Our goal is to give a new and simple proof that the B\'{e}zier polynomials
attached to a continuos function defined on a simplicial finite element, that
is an $D$ dimensional simplex, is really an approximant of the function and to
estimate the error. We stress that we are not intend to give a new proof of
some extension of the Weierstrass theorem on simplexes, rather we are
interested in better understanding the approximation mechanism in order to
handle the more complex geometries that appears in the case of nonconforming
B\'{e}zier finite element analysis. \ 

Let a compact set $\mathbf{K}\subset
\mathbb{R}^{D}$. Here $
\mathbb{R}^{D}$ is considered with standard real vector space structure and scalar
product "$\cdot$".\ We denote by $C(\mathbf{K})$ the Banach algebra of real
continuos functions on $K$ with the topology given by the norm
\begin{equation}
\left\Vert f\right\Vert :=\underset{\mathbf{x\in K}}{\sup}\left\vert
f(\mathbf{x})\right\vert ;~f\in C(\mathbf{K})\label{supnormdefinitiongeneral}%
\end{equation}
\textbf{Remark}: Recall the Stone-Weierstrass theorem \cite{GeneralTopology}
that will be used: \ Let $\mathbf{K}$ a compact topological space and
$A\subset C(\mathbf{K})$ a subalgebra of$\ C(\mathbf{K})$ with containing
constant function. If the subalgebra $A$ contains functions that distinguish
every pair of points, that is for each pair $\mathbf{x}_{1}\mathbf{,x}_{2}%
\in\mathbf{K}$ there exits $f\in A$ such that $f(\mathbf{x}_{1})\neq
f(\mathbf{x}_{2})$ then $A$ is dense in $C(\mathbf{K})$ in the norm topology.
In the our setting $\mathbf{K}$ is a simplex from.

\subsection{ \ \ Definitions and notations}

Consider a\ general $D$ dimensional closed simplex \textbf{$T$}%
\begin{equation}
\mathbf{x}_{i}\in\mathbf{T}\subset
\mathbb{R}^{D},1\leq i\leq D\label{b0}%
\end{equation}
Denote by $C(\mathbf{T})$ the normed space of continuos functions on
$\mathbf{T}$ with uniform convergence norm
\begin{equation}
\left\Vert f\right\Vert :=\underset{\mathbf{x\in T}}{\sup}\left\vert
f(\mathbf{x})\right\vert ;~f\in C(\mathbf{T})\label{normdefinition}%
\end{equation}
For a given point $\mathbf{x}\in\mathbf{T}$ we denote by $s_{i}(\mathbf{x}%
),~i=\overline{0,D}$ \ its barycentric coordinates:
\begin{align}
\mathbf{x} &  =\sum\limits_{i=0}^{D}s_{i}(\mathbf{x})\mathbf{x}_{i}%
\label{b1}\\
\sum\limits_{i=0}^{D}s_{i}(\mathbf{x}) &  =1;~s_{i}(\mathbf{x})\geq
0,~i=\overline{0,D}\label{b2}%
\end{align}
Consider only the case when the simplex is non degenerated. The correspondence
defined in Eqs. (\ref{b1}, \ref{b2}),%

\[
\mathbb{R}^{D}\ni\mathbf{x}\rightarrow(s_{0}(\mathbf{x}),\ s_{2}%
(\mathbf{x}),\cdots~,s_{D}(\mathbf{x})):=S(\mathbf{x})\in\mathbb{R}^{D+1}%
\]
transform the simplex \textbf{$T$} in\ \textbf{$T_{0}$,} the standard $D$
dimensional simplex from~$
\mathbb{R}^{D+1}$,~that is the convex hull of the unit vectors from \ $\mathbb{R}^{D+1}%
$. The map $S$~is one to one; denote its inverse by $R$:%
\begin{align}
\mathbb{R}^{D+1}  &  \supset\mathbf{T}_{0}\ni(t_{0},t_{1},\cdots,t_{D})\overset
{R}{\rightarrow}\mathbf{x}=\sum\limits_{i=0}^{D}t_{i}\mathbf{x}_{i}%
:=R(t_{0},t_{1},\cdots,t_{D})\in\mathbf{T}\label{b4}\\
\sum\limits_{i=0}^{D}t_{i}  &  =1;~t_{i}\geq0,~i=\overline{0,D} \label{b5}%
\end{align}

The one dimensional analogue of the Eq.(\ref{b2}) was the key starting point
in the probabilistic proof of the classical Weierstrass theorem by S.
Bernstein in the case of functions of one variable \cite{Sinai}, \cite{Renyi}.

In analogy to the single variable case of Bernstein polynomial attached to the
1-simplex $[0,1]$,
\[
B_{k}^{n}(s)=\frac{n!}{k!(n-k)!}s^{k}(1-s)^{n-k};~0\leq k\leq
n~;~k,n~\operatorname{integer}~
\]
the Bernstein-B\'{e}zier polynomials \cite{Bezier}, \cite{FarinCAGD},
\cite{Farin} attached to the $D$ dimensional simplex $\mathbf{T}$ from
(\ref{b0}) are defined as follows. First, we use the notation $\mathbf{k}$ for
the sequence of $D+1$ integers $(k_{0},k_{1},\cdots,k_{D})\mathbf{\in}
\mathbb{N}
^{\times(D+1)}$ and denote
\[
\left\vert \mathbf{k}\right\vert :=\sum\limits_{j=0}^{D}k_{j}%
\]
We use the standard notation for the Newton multinomial coefficients with
$\mathbf{k\in}
\mathbb{N}^{\times(D+1)}$
\begin{align*}
\binom{n}{k_{0},k_{1},\cdots,k_{D}}  &  =\binom{n}{\mathbf{k}}:=\frac
{n!}{\prod\limits_{j=0}^{D}k_{j}!}\\
\left\vert \mathbf{k}\right\vert  &  =n
\end{align*}
With these notations the B\'{e}zier polynomials of order $\left\vert
\mathbf{k}\right\vert =n$, defined on the simplex $\mathbf{T}\subset
\mathbb{R}^{D}$, indexed by $\mathbf{k}$, are defined as follows
\begin{align}
B_{\mathbf{k}}^{n}(\mathbf{x})  &  :=\binom{n}{\mathbf{k}}\prod\limits_{j=0}%
^{D}\left[  s_{j}(\mathbf{x})\right]  ^{k_{j}}\label{b6}\\
\left\vert \mathbf{k}\right\vert  &  =n;~\ 0\leq k_{j}\leq n;~j=\overline
{0,D};~k_{j}\in
\mathbb{N}
\label{b7}%
\end{align}

In the following, we will denote by $M_{n}$, the set whose elements are
sequence of integers with property:
\begin{equation}
\mathbb{\ }
\mathbb{N}^{\times(D+1)}\supset M_{n}:=\{\mathbf{k}|\mathbf{k}\in
\mathbb{N}
^{\times(D+1)};~\left\vert \mathbf{k}\right\vert =n\} \label{b7.1}%
\end{equation}
This subset $M_{n}$ of sequences appears in the Newton multinomial formula:%
\begin{equation}
\left(  \sum\limits_{j=0}^{D}a_{j}\right)  ^{n}=\sum\limits_{\mathbf{k}\in
M_{n}}\binom{n}{\mathbf{k}}\prod\limits_{j=0}^{D}\left[  a_{j}\right]
^{k_{j}} \label{Newton}%
\end{equation}

The set of Bernstein-B\'{e}zier polynomials of given order $n$, in the case of
a real continuous function of $D$ real variables defined on the simplex
$\mathbf{T}$
\[
\mathbb{R}^{D}\supset\mathbf{T\ni y}\rightarrow f(\mathbf{y})\in\mathbb{R}%
\]
generate a polynomial $B_{n}(f;\mathbf{y})$ in $D$ real variables \textbf{$y$%
}$=(y_{1},y_{2},\cdots y_{D})$, that is, by anticipating, an uniform
polynomial approximant of the function $f($\textbf{$y$}$)$%
\begin{align}
B_{n}(f;\mathbf{y})  &  :=\sum\limits_{\mathbf{k}\in M_{n}}f\left[  R\left(
\frac{k_{1}}{n},\frac{k_{2}}{n},\cdots,\frac{k_{D}}{n}\right)  \right]
B_{\mathbf{k}}^{n}(\mathbf{y})=\label{b8}\\
&  =\sum\limits_{\mathbf{k}\in M_{n}}f\left(  \sum\limits_{j=1}^{D}\frac
{k_{j}}{n}\mathbf{x}_{j}\right)  B_{\mathbf{k}}^{n}(\mathbf{y}) \label{b8.1}%
\end{align}
The previous equations (\ref{b8}), (\ref{b8.1}) define a sequence linear
operators $\widehat{B}_{n}f:=B_{n}(f;\mathbf{.})$, indexed by $n$, on the
space $C(\mathbf{T})$
\begin{equation}
\left(  \widehat{B}_{n}f\right)  (\mathbf{x})=B_{n}(f;\mathbf{x}%
)\ \label{firstDefBhat}%
\end{equation}

The set of points,%
\begin{equation}
C_{n}:=\left\{  R\left(  \frac{k_{1}}{n},\frac{k_{2}}{n},\cdots,\frac{k_{D}%
}{n}\right)  |\mathbf{k}\in M_{n}\right\}  \subset\mathbf{T} \label{b9}%
\end{equation}
are called the set of control points on the simplex $\mathbf{T}$. We have%
\begin{equation}
C_{n}=R\left(  M_{n}/n\right)  \label{b10}%
\end{equation}

\subsection{Convergence proof and error estimate for exponential functions}

We denote by $C(\mathbf{K})$ space of continuos functions on a compact set
$\mathbf{K}$\ in a finite dimensional real vector space, with the uniform
convergence topology and denote by $E(\mathbf{K})$ the subspace of
$C(\mathbf{K})$ generated by all finite linear combinations of the functions%
\begin{equation}
\exp\left(  \mathbf{a}\cdot\mathbf{x}\right)  ;~\mathbf{a}\in
\mathbb{R}^{D}\label{b11}
\end{equation}
We start with a proposition that simplifies the proof.

\begin{proposition}
\label{mark proposition separation}The subspace $E(\mathbf{K})$ contains
functions that distinguish \cite{GeneralTopology} every pair of points
\ $\mathbf{b}_{1},\mathbf{b}_{2}\in\mathbf{K}.$
\end{proposition}

\begin{proof}
We have to prove that there exists a function $g\in E(\mathbf{K})$ such that
$g(\mathbf{b}_{1})\neq g(\mathbf{b}_{2})$. Consider
\[
\mathbf{a=b}_{2}-\mathbf{b}_{1}\
\]
The function $g($\textbf{$x$}$)$ defined by%
\[
g(\mathbf{x}):=\exp\left[  \left(  \mathbf{b}_{2}-\mathbf{b}_{1}\right)
\cdot(\mathbf{x-b}_{1})\right]
\]
is also contained in $E(\mathbf{K})$ and $g(\mathbf{b}_{2})>g(\mathbf{b}%
_{1})\ $that completes the proof.
\end{proof}

Starting from this result we can justify our main lemma, which states that for
every $f\in C(\mathbf{K})$ and $\varepsilon>0$ there exists an exponential
polynomial of the form
\begin{equation}
P_{\varepsilon}(\mathbf{x})=\sum\limits_{i=1}^{N(\varepsilon)}c_{i}%
^{\varepsilon}\exp\left(  \mathbf{a}_{i}^{\varepsilon}\cdot\mathbf{x}\right)
\label{b11.1}%
\end{equation}
such that
\begin{equation}
~\left\Vert f-P\right\Vert :=\underset{\mathbf{x\in K}}{\sup}\left\vert
f(\mathbf{x})-P(\mathbf{x})\right\vert \leq\frac{\varepsilon}{2} \label{b11.2}%
\end{equation}
where we used the notation Eq.(\ref{normdefinition}). This property can be
reformulated as follows:

\begin{lemma}
\label{marker lemma density exponents}The subspace $E(\mathbf{K})$ is dense in
$C(\mathbf{K})$ in the uniform convergence topology on the compact set
$\mathbf{K}$.
\end{lemma}

\begin{proof}
The subspace $E(\mathbf{K})$ is closed under multiplication, contains identity
element so it is a Banach subalgebra of $C(\mathbf{K})$ with identity, and
according to the previous Proposition \ref{mark proposition separation},
distinguishes all of the points on the \textbf{compact set} $\mathbf{K}$. So,
according to Stone-Weierstrass theorem \ \cite{GeneralTopology},
$E(\mathbf{K})$ is dense in $C(\mathbf{K})$.
\end{proof}

Now, according with the previous Lemma, in order to prove that $B_{f}^{n}%
($\textbf{$x$}$)\overset{n\rightarrow\infty}{\rightarrow}f($\textbf{$x$}$)$,
it is sufficient to prove the convergence on the generators $\exp\left(
\mathbf{a}\cdot\mathbf{x}\right)  $ of the space $E(\mathbf{K})$. Denote%
\begin{equation}
e_{\mathbf{a},n}(\mathbf{x}):=B^{n}(f;\mathbf{x});~f(\mathbf{x})\equiv
\exp\left(  \mathbf{a}\cdot\mathbf{x}\right)  \label{12}%
\end{equation}
For the sake of clarity we decompose the proofs in several steps.

\subsubsection{The B\'{e}zier polynomial for $\exp\left(  \mathbf{a}%
\cdot\mathbf{x}\right)  $ for $\ D$ dimensional simplex, error estimate}

\begin{proposition}
\label{marker Prop exact formula exp funct}The B\'{e}zier polynomial
associated to $D$ dimensional simplex for exponential function from Eq.
(\ref{12}) can be expressed as follows%
\begin{equation}
e_{\mathbf{a},n}(\mathbf{x})=\left[  \sum\limits_{j=1}^{D}s_{j}(\mathbf{x}%
)\exp\left(  \frac{\mathbf{a}\cdot\mathbf{x}_{j}}{n}\right)  \right]
^{n}\label{13}%
\end{equation}

\end{proposition}

\begin{proof}
We use here a shorter notation: $s_{j}($\textbf{$x$}$)$ denote simply as
$s_{j}$.~From Eqs. (\ref{b8.1}, \ref{b6}, \ref{b11} \ref{b7}) and Newton
formula Eq.(\ref{Newton}) results
\begin{align}
e_{\mathbf{a},n}(\mathbf{x}) &  =\sum\limits_{\mathbf{k}\in M_{n}}\exp\left[
a\cdot\left(  \sum\limits_{j=1}^{D}\frac{k_{j}}{n}\mathbf{x}_{j}\right)
\right]  B_{i,j,k}^{n}(\mathbf{x})\label{14}\\
&  =\sum\limits_{\mathbf{k}\in M_{n}}\prod\limits_{j=0}^{D}\left[  \exp\left(
\frac{\mathbf{a}\cdot\mathbf{x}_{j}}{n}\right)  \right]  ^{k_{j}}\ \binom
{n}{\mathbf{k}}\prod\limits_{j=0}^{D}\left[  s_{j}\right]  ^{k_{j}}\nonumber\\
&  =\left[  \sum\limits_{j=1}^{D}s_{j}\exp\left(  \frac{\mathbf{a}%
\cdot\mathbf{x}_{j}}{n}\right)  \right]  ^{n}\nonumber
\end{align}
where was used the definition (\ref{b7.1}) and the multinomial Newton formula.
\end{proof}

By using Eq. (\ref{13}) we have the following

\begin{proposition}
\label{marker prop exp approximation}In the limit of large $n$ we have the
following convergence result of the B\'{e}zier polynomial $e_{\mathbf{a},n}%
($\textbf{$x$}$)$ on the simplex $\mathbf{T}$
\begin{equation}
\left\vert e_{\mathbf{a},n}(\mathbf{x})-\exp(a\cdot\mathbf{x})\right\vert
\leq\left[  \frac{K}{n}+\mathcal{O}\left(  \frac{1}{n^{2}}\right)  \right]
\exp(a\cdot\mathbf{x}) \label{asimpt4}%
\end{equation}
where $K$ is a constant, independent of $n$ and \textbf{$x$}.
\end{proposition}

\begin{proof}
By separating an $\mathcal{O}(1/n^{2})$ term
\[
\exp\left(  \frac{\mathbf{a}\cdot\mathbf{x}_{j}}{n}\right)  =1+\frac
{\mathbf{a}\cdot\mathbf{x}_{j}}{n}+\left[  \exp\left(  \frac{\mathbf{a}%
\cdot\mathbf{x}_{j}}{n}\right)  -1-\frac{\mathbf{a}\cdot\mathbf{x}_{j}}%
{n}\right]
\]
with the use of Eq.(\ref{b2}), the term in the right hand side from
Eq.(\ref{13}) we rewrite as follows%
\begin{align}
\sum\limits_{j=1}^{D}s_{j}(\mathbf{x})\exp\left(  \frac{\mathbf{a}%
\cdot\mathbf{x}_{j}}{n}\right)   &  =1+\sum\limits_{j=1}^{D}s_{j}%
(\mathbf{x})\frac{\mathbf{a}\cdot\mathbf{x}_{j}}{n}+r_{n}(\mathbf{x}%
)\nonumber\\
&  =1+\frac{\mathbf{a}\cdot\mathbf{x}}{n}+r_{n}(\mathbf{x})\label{asimpt6}%
\end{align}
where in the last equality we used Eq.(\ref{b1}) and we denoted the residual
$O(1/n^{2})$ term as
\begin{equation}
r_{n}(\mathbf{x}):=\sum\limits_{j=1}^{D}s_{j}(\mathbf{x})\exp\left(
\frac{\mathbf{a}\cdot\mathbf{x}_{j}}{n}\right)  -1-\frac{\mathbf{a}%
\cdot\mathbf{x}}{n}\label{asimpt7}%
\end{equation}
For large $n$, by using the remainder formula for Taylor series we have the
inequality
\begin{align}
\left\vert r_{n}(\mathbf{x})\right\vert  &  \leq\frac{K_{1,n}}{n^{2}%
}\label{asimpt8}\\
K_{1,n} &  =\frac{1}{2}\sum\limits_{j=1}^{D}(\mathbf{a}\cdot\mathbf{x}%
_{j})^{2}\exp\left(  \frac{\mathbf{a}\cdot\mathbf{x}_{j}}{n}\right)
<C\label{asimpt9}%
\end{align}
where the constant $C$ can be optimized by suitable coordinate change. From
Eqs. (\ref{13}, \ref{asimpt6}) results%
\begin{equation}
e_{\mathbf{a},n}(\mathbf{x})=\left[  1+\frac{\mathbf{a}\cdot\mathbf{x}}%
{n}+r_{n}(\mathbf{x})\right]  ^{n}\label{asimpt10}%
\end{equation}
In order to find the speed of convergence by using Eqs. (\ref{asimpt8},
\ref{asimpt9}) we compute%
\begin{align}
\left\vert \log\frac{e_{\mathbf{a},n}(\mathbf{x})}{\exp(\mathbf{a}%
\cdot\mathbf{x})}\right\vert  &  <\frac{K}{n}+\mathcal{O}\left(  \frac
{1}{n^{2}}\right)  \label{asimpt11}\\
K &  =C+\frac{1}{2}\underset{\mathbf{x}\in\mathbf{T}}{\max}(\mathbf{a}%
\cdot\mathbf{x})\label{asimpt12}%
\end{align}
From Eq.(\ref{asimpt11}) we obtain the relative error bound%
\begin{equation}
\left\vert \frac{e_{\mathbf{a},n}(\mathbf{x})-\exp(\mathbf{a}\cdot\mathbf{x}%
)}{\exp(\mathbf{a}\cdot\mathbf{x})}\right\vert <\frac{K}{n}+\mathcal{O}\left(
\frac{1}{n^{2}}\right)  \label{asimpt13}%
\end{equation}
that completes the proof.

\begin{corollary}
\label{markerCorolaryUnifConvergence}In the special case when $f(\mathbf{x}%
)=\exp(\mathbf{a}\cdot\mathbf{x})$ \ we have the uniform convergence in the
simplex $\mathbf{T}$%
\[
\left\Vert B_{n}(f;\mathbf{.})-f(\mathbf{.})\right\Vert :=\underset
{\mathbf{x\in T}}{\sup}\left\vert B_{n}(f;\mathbf{x})-f(\mathbf{x})\right\vert
\leq\left[  \frac{K}{n}+\mathcal{O}\left(  \frac{1}{n^{2}}\right)  \right]
\underset{\mathbf{x\in T}}{\sup}\exp(\mathbf{a}\cdot\mathbf{x})
\]
or by notation (\ref{firstDefBhat})
\[
\left\Vert \widehat{B}_{n}f-f\right\Vert \leq\left[  \frac{K}{n}%
+\mathcal{O}\left(  \frac{1}{n^{2}}\right)  \right]  \left\Vert f\right\Vert
\]
For all \thinspace exponential polynomials $P(\mathbf{x}):=\sum\limits_{i=1}%
^{N}c_{i}\exp(\mathbf{a}_{i}\cdot\mathbf{x})$ and every $\varepsilon$ there
exists an $N(P,\varepsilon)$ such that for all $n\geq N,$ \
\begin{equation}
\left\Vert \widehat{B}_{n}P-P\right\Vert \leq\frac{\varepsilon}{4}%
\label{expoaproxineg}%
\end{equation}

\end{corollary}
\end{proof}

\subsection{Convergence proof, and error estimate, general case.}

By \ Eqs. (\ref{b8}, \ref{b8.1}) we defined a family of linear operators
$B^{n}(f;\mathbf{y})$ on the space of continuos functions $C(\mathbf{T})$,
indexed by integer $n$, that assign to every $f\in C(\mathbf{T})$ a polynomial
of degree $n$. These family of operators has the following properties%
\begin{align}
f(\mathbf{x}) &  \equiv1~\Rightarrow~B_{n}(f;\mathbf{x})\equiv1\label{conv1}\\
f(\mathbf{x}) &  >g(\mathbf{x})~\Rightarrow~B_{n}(f;\mathbf{x})>B_{n}%
(f;\mathbf{x})\label{conv2}%
\end{align}
The property (\ref{conv1}) results from Eqs. (\ref{b2}, \ref{b6},
\ref{Newton}, \ref{b8.1}), while Eq.(\ref{conv2}) result from (\ref{b6}) (the
B\'{e}zier polynomials are positive) and Eq.(\ref{b8.1}). From Eqs.
(\ref{conv1}, \ref{conv2}) results that
\begin{align}
\left\Vert f\right\Vert  &  :=\underset{\mathbf{x\in T}}{\sup}\left\vert
f(\mathbf{x})\right\vert \leq A~\Rightarrow\nonumber\\
\left\Vert B_{n}(f;\mathbf{.})\right\Vert  &  :=\underset{\mathbf{x\in T}%
}{\sup}\left\vert B_{n}(f;\mathbf{x})\right\vert \leq A\label{conv4}%
\end{align}
The last relation means that the operator $\widehat{B}_{n}$ defined by
Eq.(\ref{firstDefBhat}) on $C(\mathbf{T})$ is a bounded linear operator, whose
norm is less or equal to $1$, and by Eq.(\ref{conv1}) its norm is $1$,
attained by constant functions:
\begin{align*}
\left\Vert \widehat{B}_{n}\right\Vert  &  =1\\
\left\Vert \widehat{B}_{n}f\right\Vert  &  \leq\left\Vert f\right\Vert
\end{align*}
Consequently
\begin{equation}
f\in C(\mathbf{T})~~\Rightarrow\left\Vert (\widehat{B}_{n}-\widehat
{1})f\right\Vert \leq2~\left\Vert f\right\Vert \label{boundBhat}%
\end{equation}
where $\widehat{1}$ is the unit operator on $C(\mathbf{T})$.\ Now we can state
and prove our main result :

\begin{theorem}
For all $f\in C(\mathbf{T})$ we have%
\begin{equation}
\underset{n\rightarrow\infty}{\lim}\underset{\mathbf{x\in T}}{\sup}\left\vert
f(\mathbf{x})-B_{n}(f;\mathbf{x})\right\vert =0 \label{theorem1}%
\end{equation}
which means that the polynomials $B_{n}(f;\mathbf{x})$ formed from B\'{e}zier
polynomials approximate uniformly the function $f\in C(\mathbf{T})$. In term
of linear operator notation the equivalent form is%
\begin{equation}
\underset{n\rightarrow\infty}{\lim}\left\Vert \widehat{B}_{n}f-f\right\Vert =0
\label{theorem2}%
\end{equation}

\end{theorem}

\begin{proof}
Let $f\in C(\mathbf{T})$ \ and $\varepsilon>0$.$\ $Then by Lemma
(\ref{marker lemma density exponents}) there exists an exponential polynomial
$P_{\varepsilon}(\mathbf{x)}\in E(\mathbf{T})$ of the form
\[
\sum\limits_{i=1}^{N}c_{i}\exp(\mathbf{a}_{i}\cdot\mathbf{x})
\]
such that%
\begin{equation}
\left\Vert f-P_{\varepsilon}\right\Vert \leq\frac{\varepsilon}{4} \label{ff1}%
\end{equation}
By Corollary \ref{markerCorolaryUnifConvergence} and Eq. (\ref{expoaproxineg}%
), there exists an $N:=N(P_{\varepsilon},\varepsilon)$ such that for all
$n\geq N(P_{\varepsilon},\varepsilon)$%
\begin{equation}
\left\Vert (\widehat{B}_{n}-\widehat{1})P_{\varepsilon}\right\Vert \leq
\frac{\varepsilon}{2} \label{ff2}%
\end{equation}
By using the triangle inequality for norm and Eqs. (\ref{boundBhat},
\ref{ff1}, \ \ref{ff2}) we obtain%
\begin{align*}
\left\Vert (\widehat{B}_{n}-\widehat{1})f\right\Vert  &  =\left\Vert
(\widehat{B}_{n}-\widehat{1})\left(  f-P_{\varepsilon}\right)  +(\widehat
{B}_{n}-\widehat{1})P_{\varepsilon}\right\Vert \leq\\
\left\Vert (\widehat{B}_{n}-\widehat{1})\left(  f-P_{\varepsilon}\right)
\right\Vert +\left\Vert (\widehat{B}_{n}-\widehat{1})P_{\varepsilon
}\right\Vert  &  \leq\left\Vert (\widehat{B}_{n}-\widehat{1})\right\Vert
\left\Vert \left(  f-P_{\varepsilon}\right)  \right\Vert +\frac{\varepsilon
}{2}<\varepsilon
\end{align*}
that completes the proof.\ 
\end{proof}

\section{Conclusions}

We presented a new proof of the convergence of the B\'{e}zier polynomials
attached to a simplex in $D$ dimensions. We proved that the relative error is
dominated by an asymptotic term of the form $O(dk/n)$ where $n$ is the order
of the B\'{e}zier polynomial, $d$ is the largest side of the simplex and $k$
is them largest wave number that appear in the Fourier expansion of the
function to be approximated.

\textbf{Acknowledgement}

This work has been carried out within the framework of the EUROfusion
Consortium as a complementary project and has been received funding from the
Romanian National Education Minister / Institute of Atomic Physics under
contract 1EU-2/2//01.07.2016. The collaboration with the C\u{a}lin Vlad
Atanasiu is acknowledged.

\end{document}